\documentclass[12pt]{amsart}
\usepackage{bbm}
\usepackage[pdftex]{graphicx}
\usepackage{amscd, amssymb, dsfont, upgreek, mathrsfs}
\input{xy}
\xyoption{all}

\newtheorem{Thm}{Theorem}
\newtheorem{Conj}[Thm]{Conjecture}
\newtheorem{Prop}[Thm]{Proposition}
\newtheorem{Def}[Thm]{Definition}
\newtheorem{Def/Thm}[Thm]{Definition/Theorem}
\newtheorem{Cor}[Thm]{Corollary}
\newtheorem{Lemma}[Thm]{Lemma}

\theoremstyle{remark}

\newcommand{\ti }{\times}
\newcommand{\ot }{\otimes}

\newcommand{\NN}{{\mathbb N}}

\newcommand{\PP }{{\mathbb P}}

\newcommand{\QQ }{{\mathbb Q}}
\newcommand{\CC }{{\mathbb C}}
\newcommand{\ZZ }{{\mathbb Z}}

\newcommand{\lan}{\langle}
\newcommand{\ran}{\rangle}

\begin{document}

\title[Crepant resolution conjecture for $\CC^5/\ZZ_5$]
{Crepant resolution conjecture for $\CC^5/\ZZ_5$}

\author{Hyenho Lho}
\address{Department of Mathematics, ETH Z\"urich}
\email {hyenho.lho@math.ethz.ch}

\begin{abstract} 
We study the relationship between Gromov-Witten invariants of local $\PP^4$ and Gromov-witten invariants of $[\CC^5/\ZZ_5]$ for all genera. We state the crepant resolution conjecture in explicit form and prove this conjecture for $g=2,3.$

 \end{abstract}

\maketitle

\setcounter{tocdepth}{1} 
\tableofcontents

\setcounter{section}{-1}

\section{Introduction}

\subsection{Crepant resolution conjecture} Let $\mathcal{A}$ be an algebraic orbifold and denote by $A$  the coarse moduli space of $\mathcal{A}$. Let 
$$B\rightarrow A$$
be a crepant resolution. Then crepant resolution conjecture state the relationship between Gromov-Witten theory of $\mathcal{A}$ and $B$. The conjecture was verified in many cases in genus $0$, see \cite{BG,BT,Co,CIT}. For higher genus, one can also state the crepant resolution conjecture using Givental's quantization formalism, see \cite{CR}.

The total space $K\PP^4$ of canonical bundle over $\PP^4$ is well-known to be a crepant resolution of the quotient,
\begin{align*}
    \epsilon_5 : K\PP^4 \rightarrow \CC^5/\ZZ_5.
\end{align*}

In this paper, we study the relationship between Gromove-Witten theory of $K\PP^4$ and $[\CC^5/\ZZ_5]$ for all genera. In our case, crepant resolution conjecture is stated in very simple and explicit form, see Conjecture \ref{Conj1}. Using Givental-Teleman's classification theorem and Tseng's orbifold quantum Riemann-Roch theorem, we can reduce the crepant resolution conjecture to very explicit equations related to Picard-Fuchs equation of local $\PP^4$ and Bernoulli polymonials, see Proposition \ref{Prop1}.  
It is interesting question to find out how our conjecture here is related to the version with Givental's quantization formalism. 

Our conjecture can be stated for the situation
$$\epsilon_n : K\PP^{n-1} \rightarrow [\CC^n/\ZZ_n]$$
for all $n \in \NN$. The case $n=2$ was studied in \cite{LP} and \cite{LP2}. The case in our paper is $n=5$. The case $n=5$ is especially interesting because the Gromov-Witten theory of $K\PP^4$ and quintic threefold is closely related. For example, Gromov-Witten potential functions of both theories are expected to lie in the same ring up to some small modifications of sign. Furthermore one might hope to formulate similar conjecture for quintic threefold. We will come back to this problems in the future.

\subsection{Gromov-Witten theory for local $\PP^4$} Gromovw-Witten theory of $K\PP^4$ can be studied by twisted theories associated to $\PP^4$ as follows. Let the algebraic tors
$$\mathsf{T}=(\CC^*)^5$$
act with the standard linearization on $\PP^4$ with weights $\lambda_0,\dots,\lambda_4$ on the vector space $H^0(\PP^4,\mathcal{O}_{\PP^4}(1))$.

Let $\overline{M}_g(\PP^n,d)$ be the moduli space of stable maps to $\PP^4$  equipped with the canonical $\mathsf{T}$-action, and let
\begin{align*}
    \mathsf{C} \rightarrow \overline{M}_g(\PP^4,d), \, f : \mathsf{C} \rightarrow \PP^4, \, \mathsf{S}=f^*\mathcal{O}_{\PP^4}(-1)\rightarrow \mathsf{C}
\end{align*}
be the standard univeral structures.
The Gromov-Witten invariants of the twisted geometry of $\PP^4$ are defined via the equivariant integrals

\begin{align}\label{twt}
 \lan \gamma_1,\dots,\gamma_n \ran^{K\PP^4}_{g,n,d} = \int_{[\overline{M}_g(\PP^4,d)]^{vir}}
e(-R\pi_*\mathsf{S}^{5}) \prod_{i=1}^n \text{ev}_i^*(\gamma_i)
\, .
\end{align}
The integral \eqref{twt} is homogenenous of degree $0$ in localized equivariant cohomology and defines a rational number
$$\lan \gamma_1,\dots,\gamma_n \ran^{K\PP^4}_{g,n,d} \in \QQ $$
after the specialization
\begin{align}\label{Specialization}
\lambda_i=\zeta^i
\end{align}
for a primitive fifth root of unity $\zeta^5=1.$ Throughout the paper we always consider $\mathsf{T}$-equivariant theory after specialization \eqref{Specialization}.
We define series for the local $K\PP^4$ geometry,
\begin{align*}
    \lan\lan \gamma_1,\dots,\gamma_m\ran\ran^{K\PP^4}_{g,m}:=\sum_{d=0}^\infty \frac{t^d}{d!}\lan\gamma_1,\dots,\gamma_m\ran^{K\PP^4}_{g,m,d}\,.
\end{align*}

\subsection{Gromov-Witten theory for $[\CC^5/\ZZ_5]$} The inertia stack $I[\CC^5/\ZZ_5]$ has five components corresponding to the five element $\{1,\omega,\omega^2,\omega^3,\omega^4\}$ of $\ZZ_5$. Since each component is contractable, the graded vector space

\begin{align*}
    H^*_{\text{orb}}([\CC^5/\ZZ_5])=H^*(I[\CC^5/\ZZ_5])
\end{align*}
has a canonical basis $\{\phi_0,\phi_1,\phi_2,\phi_3,\phi_4\}$ corresponding to the five components.

Let the algebraic torus 
$$\mathsf{T}$$
act standardly with wights $\lambda_0,\dots,\lambda_4$ on $\CC^5$.

Let $\overline{M}_{g,m}^{orb}([\CC^5/\ZZ_5])$ be the moduli space of orbifold stable maps to $[\CC^5/\ZZ_5]$ equipped with the canonical $\mathsf{T}$-actions. For $n=n_0+n_1+\dots+n_4$, the Gromov-Witten invariants of $[\CC^5/\ZZ_5]$ are defined via the equivariant integrals
\begin{align*}
 &\lan \phi_0^{n_0},\phi_1^{n_1},\dots,\phi_4^{n_4}\ran^{[\CC^5/\ZZ_5]}_{g,n}= \\ 
 &\int_{[\overline{M}^{orb}_{g,n}([\CC^5/\ZZ_5])]^{vir}}\prod_{i=1}^{n_0}\text{ev}_i^*(\phi_0)\prod_{i=n_0+1}^{n_0+n_1}\text{ev}_i^*(\phi_1)\prod_{i=n_0+n_1+1}^{n_0+n_1+n_2}\text{ev}_i^*(\phi_2)\\
 &\prod_{i=n_0+n_1+n_2+1}^{n_0+n_1+n_2+n_3}\text{ev}_i^*(\phi_3)\prod_{i=n_0+n_1+n_2+n_3+1}^{n_0+n_1+n_2+n_3+n_4}\text{ev}_i^*(\phi_4).
\end{align*}

\noindent The above integral defines a rational number
$$\lan \phi_0^{n_0},\phi_1^{n_1},\dots,\phi_4^{n_4}\ran^{[\CC^5/\ZZ_5]}_{g,n} \in \QQ$$
after the specialization
\begin{align}\label{Specialization2}
    \lambda_i=\zeta^i. 
\end{align}
    Here also we consider $\mathsf{T}$-equivariant theory after the specialization \eqref{Specialization2}. We define series for the $[\CC^5/\ZZ_5]$ geometry,
\begin{align*}
    \lan\lan \gamma_1,\dots,\gamma_m \ran\ran^{[\CC^5/\ZZ_5]}_{g,m}:=\sum_{k=0}^\infty \frac{s^k}{k!}\lan \gamma_1,\dots,\gamma_m,\phi_1,\dots,\phi_k\ran^{[\CC^5/\ZZ_5]}_{g,m+k}.
\end{align*}

\subsection{Acknowledgments} 
I am very grateful to H.~Fan, B.~Kim, R.~Pandharipande, J.~Shen
for discussions 
about Gromov-Witten theory and crepant resolution conjecture. I was supported by the grant ERC-2012-AdG-320368-MCSK.

\

\section{Cohomological field theory}

\subsection{Definitions.} The notion of a cohomological field theory (CohFT) was introduced in \cite{KM,Ma}. We review the treatment of \cite{PPZ}.

Let $\mathsf{A}$ be a commutative $\CC$-algebra. Let $\mathsf{V}$ be a free $\mathsf{A}$-module of finite rank, let
$$ \eta : \mathsf{V}\otimes\mathsf{V}\rightarrow \mathsf{A}$$
be an non-degenerate symmetric pairing, and let $\mathsf{1}\in\mathsf{V}$ be a distinguished element. Let $\{e_i\}$ be a basis of $\mathsf{V}$ and denote by
$$\eta_{jk}=\eta(e_j,e_k)$$
the matrix of symmetric form. The inverse matix is denoted by $\eta^{jk}$.

A cohomological field theory consists of a data $\Omega=(\Omega_{g,r})_{2g-2+r>0}$ of elements
$$\Omega_{g,r}\in H^*(\overline{\mathcal{M}}_{g,r},\mathsf{A}) \otimes (\mathsf{V}^*)^{\otimes r}$$

We view $\Omega_{g,r}$ as associating a cohomology class on $\overline{\mathcal{M}}_{g,r}$ to elements of $\mathsf{V}$ assigned to the $r$ markings. $\Omega$ satisfy the following CohFT axioms.
\begin{enumerate}
 \item[(i)] Each $\Omega_{g,n}$ is $S_r$-invariant, where the action of the symmetric group $S_r$ permutes both the marked points of $\overline{M}_{g,r}$ and the copies of $\mathsf{V}^*$.
 \item[(ii)] Let $q$ and $\tilde{q}$ be the canonical gluing maps:
 \begin{align*}
     q &: \overline{M}_{g-1,r+2}\rightarrow \overline{M}_{g,r}\,, \\
     \tilde{q} &: \overline{M}_{g_1,r_1+1}\times\overline{M}_{g_2,r_2+1}\rightarrow \overline{M}_{r,g}\,.
 \end{align*}
The pull-backs $q^*(\Omega_{g,r})$ and $\tilde{q}^*(\Omega_{g,r})$ are equal to the contractions of 
$$\Omega_{g-1,r+2}\, \text{and}\, \Omega_{g_1,r_1+1}\otimes\Omega_{g_2,r-2+1} $$
by the bi-vector
$$\sum_{j,k} \eta^{jk} e_j\otimes e_k$$
inseted at the two glued points.
\item[(iii)] Let $p$ be the canonical map which forget the last marking:
$$p : \overline{M}_{g,r+1}\rightarrow \overline{M}_{g,r}.$$
For $v_1,\dots,v_r \in\mathsf{V}$ we require
\begin{align*}
    \Omega_{g,r+1}(v_1\ot\dots\ot v_r\ot \mathsf{1})&=p^*\Omega_{g,r}(v_1\ot\dots\ot v_r)\, ,\\
    \Omega_{0,3}(v_1\ot v_2\ot\mathsf{1})&=\eta(v_1,v_2) \, .
\end{align*}
\end{enumerate}

\begin{Def}
A data $\Omega=(\Omega_{g,r})_{2g-2+r>0}$ of elements
$$\Omega_{g,r}\in H^*(\overline{M}_{g,r},\mathsf{A})\ot (\mathsf{V}^*)^{\ot r}$$
satisfying properiest (i),(ii) and (iii) is called a cohomological field thoery with unit (CohFT).
\end{Def}

A CohFT $\omega$ composed only of degree $0$ classes,
$$\omega\in H^0(\overline{\mathcal{M}}_{g,r},\mathsf{A})\ot(\mathsf{V}^*)^{\ot r},$$
is called a {\em topological field theory}. By propety (ii), a topologycial field thoery is uniquely determined by the associated quantum product.

A CohFT $\Omega$ yields a quantum product $\bullet$ on $\mathsf{V}$ by following.
\begin{align}\label{QP}
    \eta(v_1\bullet v_2,v_3)=\Omega_{0,3}(v_1\ot v_2\ot v_3).
\end{align}
By (ii), $\bullet$ is associative. By (iii), the distinguished element $\mathsf{1}$ is the identity for $\bullet$. 

\subsection{Semisimplicity.}
\subsubsection{Classification.} Let $\Omega$ be a CohFT with respect to $(\mathsf{V},\eta,\mathsf{1})$. We call $\Omega$  semisimple, if the associated quantum product $\bullet$ defined by \eqref{QP} is semisimple. 

We review here the classification of semisiple CohFT, see \cite{Te, SS}.
Let $\mathsf{R}$ be the End$(\mathsf{V})$-valued power series 
$$\mathsf{R}(z)=1+R_1 z+R_2 z^2+\dots \in \mathsf{Id}+z \text{End}(\mathsf{V})[[z]],$$
satisfying the symplectirc condition,
$$\mathsf{R}(z)\mathsf{R}^*(-z)=\mathsf{Id}\, , $$
where $\mathsf{R}^*$ is the adjoint with respect to $\eta$ and $\mathsf{Id}$ is the identity matrix.

\subsubsection{Actions on CohFTs.} Let $\Omega=(\Omega_{g,r})$ be a CohFT with respect to ($\mathsf{V},\eta,\mathsf{1}$). Fix a symplectic matrix
$$\mathsf{R}(z)\mathsf{R}^*(-z)=\mathsf{Id} $$
as above. A new CohFT with respect to $(\mathsf{V},\eta,\mathsf{1})$ is obtained via the cohomology elements
$$ \mathsf{R}\Omega=(\mathsf{R}\Omega)_{g,r} \, ,$$
defined as sums over stable graph $\Gamma$ of genus $g$ with $r$ legs, with contributions coming from vertices, legs and edges.
\begin{align*}
    (\mathsf{R}\Omega)_{g,r}=\sum_{\Gamma}\frac{1}{\text{Aut}(\Gamma)}(\iota_\Gamma)_*\left(\prod_v \text{Cont}(v) \prod_e \text{Cont}(e) \prod_l \text{Cont}(l) \right)
\end{align*}
where:
\begin{enumerate}
 \item[(i)] the vertex contribution is
 $$\text{Cont}(v)=\Omega_{g(v),r(v)} $$
 with $g(v)$ and $r(v)$ denoting the genus and number of half-edges and legs of the vertex,
 \item[(ii)] the leg contribution is 
 $$\text{Cont}(l)=\mathsf{R}(\psi_l), $$
 where $\psi_l$ is the cotangent class at the marking corresponding to the leg,
 \item[(iii)] the edge contribution is
 $$\text{Cont}(e)=\frac{\eta^{-1}-\mathsf{R}(\psi'_e)\eta^{-1}\mathsf{R}(\psi''_e)^T}{\psi'_e+\psi''_e},$$
 where $\psi'_e$ and $\psi''_e$ are the cotangent classes at the node which represents the edge $e$. The edge contribution is well-defined by the symplectic condition.
\end{enumerate}

A second action on CohFTs is given by translations. Let $\Omega$ be a CohFT with respect to $(\mathsf{V},\eta,\mathsf{1})$ as before and let 
$$\mathsf{T}(z)=T_2 z^2+T_3 z^3+\dots$$
be a $\mathsf{V}$-valued power series with vanishing coefficients in degree $0$ and $1$.
Define a new CohFT $\mathsf{T}\Omega$ with respect to $(\mathsf{V},\eta,\mathsf{1})$ by following.
\begin{align*}
    (\mathsf{T}\Omega)_{g,r}(v_1\ot\dots\ot v_r)=\sum_{m=0}^{\infty}\frac{1}{m!}(p_m)_*\Omega_{g,r+m}(v_1\ot\dots\ot v_r\ot \mathsf{T}(\psi_{r+1})\ot\dots\ot\mathsf{T}(\psi_{r+m}))
\end{align*}
where
$$p_m : \overline{\mathcal{M}}_{g,r+m}\rightarrow\overline{\mathcal{M}}_{g,r}$$
is the forgetful morphism.

\subsubsection{Reconstruction.} With the above settings, we can state the Givental-Teleman classification theorem \cite{Te,SS}. Fix a semisimple CohFT $\Omega$ with respect to $(\mathsf{V},\eta,\mathsf{1})$ and denote by $\omega$ the degree $0$ topological part of $\Omega$. Given a symplectic matrix
$$\mathsf{R} \in \mathsf{Id}+z \text{End}(\mathsf{V})[[z]],$$
we define a power seies $\mathsf{T}$ as following:
$$\mathsf{T}(z)=z(\mathsf{Id}-\mathsf{R}(\mathsf{1}))\in\mathsf{V}[[z]].$$

\begin{Thm}(\cite{Te}, Lemma 2.2)\label{Classification} There exists a unique symplectic matrix $\mathsf{R}$ for which
$$\Omega=\mathsf{RT}\omega.$$

\end{Thm}

\

\section{$\mathsf{R}$-matrix}
We review here the properties of $\mathsf{R}$-matrix in Theorem \ref{Classification} closely following the treatment of \cite{SS,Book}.

\subsubsection{Frobenious manifold.}
\begin{Def}
 A Frobenius manifold $\mathsf{V}$ is a quadruple $(\eta,\bullet, \mathsf{A}, \mathbf{1})$ satifying the following conditions:

\begin{enumerate}
    \item $\eta$ is Riemmanian metric on  $\mathsf{M}$,
    \item $\bullet$ is commutative and associative product on $T\mathsf{M}$,
    \item $\mathsf{A}$ is a symmetric tensor,
     $$A:T\mathsf{M} \otimes T\mathsf{M} \otimes T\mathsf{M} \rightarrow \mathcal{O}_{\mathsf{M}}  $$,
    \item $\eta(X\bullet Y, Z)=A(X,Y,Z)$,
    \item $\mathbf{1}$ is a $\eta$-flat unit vector field.
\end{enumerate}
\end{Def}

\noindent For every CohFT $\mathsf{\Omega}$ with respect to $(\mathsf{V},\eta,\mathsf{1})$, the genus $0$ part of $\Omega$ naturally determines a formal Frobenius manifold structure at origin of $\mathsf{V}$.

\subsubsection{Flat coordinates}

Let $p$ be a poin of $\mathsf{M}$. As $\eta$ is flat, we can find flat coordinates $(t^0,t^1,\dots,t^{m-1})$ in a neighborhood of $p$. Denote by
$$\phi_i=\frac{\partial}{\partial t^i}$$
the corresponding flat vector fields. The convention,
$$ \mathsf{1}= \phi_0$$
will usually be followed. Let $\eta_{ij}=\eta(\phi_i,\phi_j)$, and let $\eta^{ij}$ denote the inverse matrix. By flatness, $\eta_{ij}, \eta^{ij}$ are constant matrices.

\subsubsection{Semisimple points and canonical coordinates}

A point $p\in \mathsf{M}$ is called semisimple if the tangent algebra $(T\mathsf{M},\bullet)$ is a semisimple algebra. For semisimple point $p$, we can find canonical coordinate $(u^0,u^2,\dots,u^{m-1})$ in a neighborhood of p. Denote by 
$$e_i=\frac{\partial}{\partial u^i} $$
the corresponding vector fields.
They satisfy the followings.
$$ e_i \bullet e_j = \delta _{ij} e_i.$$

\noindent Define normalized canonical basis $\tilde{e}_i$ by
$$ \tilde{e}_i = \eta(e_i,e_i)^{-\frac{1}{2}} e_i. $$

\subsubsection{The transition matrix}Let $\mathsf{\Psi}$ be the transition matrix from the basis $\phi_i$ to $\tilde{e}_\alpha$.
By the orthonormality of $\tilde{e}_\alpha$, the elements of $\mathsf{\Psi}$ are
$$\mathsf{\Psi}_{\alpha i}=g(\tilde{e}_\alpha,\phi_i).$$

\subsubsection{Fundamental solutions and $\mathsf{R}$-matrix}

We define $$\mathsf{R}(z)=\sum_{k=0}^{\infty}\mathsf{R}_k z^k$$ by following flatness euation,

\begin{align}\label{FSM}
    zd\mathsf{R}^{-1}\mathsf{\Psi}+z\mathsf{R}^{-1}d\mathsf{\Psi}-d\mathsf{U}\mathsf{R}^{-1}\mathsf{\Psi}+\mathsf{R}^{-1}d\mathsf{U}\mathsf{\Psi}.
\end{align}
where $\mathsf{U}=\text{Diag}(u^0,u^1,\dots,u^{m-1})$.

$\mathsf{R}(z)$ is uniquely determined by \eqref{FSM} up to a right multiplication by a constant matrix

$$ \text{exp}(\sum_{k \ge 1} \mathsf{a}_{2k-1}z^{2k-1})$$
where
$$\mathsf{a}_{2k-1}=\text{diag}[a^0_{0,2k-1},a^1_{1,2k-1},\dots,a^{m-1}_{m-1,2k-1}] $$
are constant diagonal matrix.

\

\section{Basic hypergeometric series}\label{BHS}
We here introduce some hypergeometric series related to Gromov-Witten thoery of local $\PP^4$ and $[\CC^5/\ZZ_5]$. 

\subsubsection{I-function}
Define I-fucntion for local $\PP^4$ and $[\CC^5/\ZZ^5]$ by following:
\begin{align*}
    I^{K\PP^4}(q,z):= \sum_{d} \frac{\prod_{k=0}^{5d-1}(-5H-kz)}{\prod_{k=1}^d(H+kz)^5}q^d\, ,\\
    I^{[\CC^5/\ZZ_5]}:=\sum_{a \ge 0} \frac{\psi^a}{z^a a!} \prod_{\substack{0 \le k < \frac{a}{5}\\ \lan k \ran =\lan \frac{a}{5}\ran}} (1-(kz)^5) \phi_a\, .
\end{align*}
We define the mirror map $t(q)$ and $s(\psi)$ for each theory as following equations.

\begin{align*}
    I^{K\PP^4}(q,z)=1+\frac{t(q)}{z}+\mathsf{O}(\frac{1}{z^2}) \, , \\
    I^{[\CC^5/\ZZ_5]}(\psi,z)=1+\frac{s(\psi)}{z}+\mathsf{O}(\frac{1}{z^2}) \, .
\end{align*}
Define $L^{K\PP^4}(q)$ and $L^{[\CC^5/\ZZ_5]}(\psi)$ by the followings.
\begin{align*}
    L^{K\PP^4}(q)=(1+5^5 q)^{-\frac{1}{5}} \, ,\\
    L^{[\CC^5/\ZZ_5]}(\psi)=-\psi(1+\frac{\psi^5}{5^5})^{-\frac{1}{5}}\, .
\end{align*}

In the below, we use $x$ to denote $q$ or $\psi$ depending on the context.
For any $F(x,z)\in \CC[[x,z^{-1}]]$, define

\begin{align*}
    D^\bullet=\left\{ \begin{array}{rl} q\frac{\partial}{\partial q} & \text{if } \bullet = K\PP^4  \\
                                    \psi\frac{\partial}{\partial \psi} & \text{if } \bullet = [\CC^5/\ZZ_5]  \, . \end{array}\right.
\end{align*}
Denote by $\mathsf{M}^\bullet$ the Birkhoff factorization operator defined by:
\begin{align*}
    \mathsf{M}^{\bullet} F(x,z):= z D^\bullet \frac{F(x,z)}{F(x,\infty)}.
\end{align*}
Define power series $\mathsf{F}^\bullet_i(x,z)$ and $C^\bullet_i(x)$ by
\begin{align*}
    \mathsf{F}^{\bullet}_i(x,z)=\mathsf{M}^i I^\bullet(x,z) \,, \, \, C^{\bullet}_i(x)=\mathsf{F}^{\bullet}_i(x,\infty)\, .
\end{align*}
The following relations were proven in \cite{ZaZi}.
\begin{Prop} The power series $C^\bullet_i(x)$ satisfy following equations.
 \begin{align*}
 &C^\bullet_0=1 \,,\,\,\, C^\bullet_2=C^\bullet_4 \,,\,\,\, {C^\bullet_1}^2 {C^\bullet_2}^2 C^\bullet_3 =(-1)^{\delta(\bullet)} {L^\bullet}^5\,,
 \end{align*}
where $$\delta(\bullet) = \left\{ \begin{array}{rl} 0 & \text{if } \bullet = K\PP^4  \\
                                    1 & \text{if } \bullet = [\CC^5/\ZZ_5]  \, . \end{array}\right.$$
\end{Prop}

In order to state the crepant resolution conjecture,
we require the following additional series in $x$.

\begin{align*}
    &X^{\bullet}:= \frac{D^{\bullet} C^{\bullet}_1}{C^{\bullet}_1} \,, \\
    &Y^{\bullet}:= \frac{D^{\bullet} C^{\bullet}_2}{C^{\bullet}_2} \,.
\end{align*}

For $K\PP^4$, define following power siries in $q$.
\begin{align}\label{B1}
    \nonumber&B_1=-5X \\
    &B_2=5^2(DX+{X}^2) \\
    \nonumber&B_3=-5^3(D^2X+3X(DX)+X^3) \\
    \nonumber&B_4=5^4(D^3X+4X(D^2X)+3(DX)^2+6{X}^2(DX)+{X}^4).
\end{align}
The following relations among $X,Y$ were proven in \cite[Section 3.2]{YY}.

\begin{align}\label{R1}
    B_4=&(1-L^5)(10B_3-35B_2+50B_1-24)  ,\\
  \nonumber  DY=&\frac{2}{5}(L^5-1)+2(L^5-1)X-2X^2-4DX\\
\nonumber    &+(L^5-1)Y-Y^2-2XY  .
\end{align}
In \eqref{B1} and \eqref{R1}, we omitted the upper subscript $K\PP^4$.

For $[\CC^5/\ZZ_5]$, define following power series in $\psi$.
\begin{align}\label{B2}
    \nonumber&B_1=\frac{1}{5}X \\
    &B_2=\frac{1}{5^2}(DX+X^2) \\
    \nonumber&B_3=\frac{1}{5^3}(D^2X+3X(DX)+X^3) \\
    \nonumber&B_4=\frac{1}{5^4}(D^3X+4X(D^2X)+3(DX)^2+6X^2(DX)+X^4).
\end{align}
Similarly we obtain the following relations.

\begin{align}\label{R2}
    B_4=&(1+\frac{L^5}{5^5})(2B_3-\frac{7}{5}B_2+\frac{2}{5}B_1-\frac{24}{625}) \, ,\\
\nonumber    DY=&-10(1+\frac{L^5}{5^5})+10(1+\frac{L^5}{5^5})X+5(1+\frac{L^5}{5^5})Y\\
\nonumber    &-2X^2-4DX-2XY-Y^2 \, .
\end{align}
In \eqref{B2} and \eqref{R2}, we omitted the upper subscript $[\CC^5/\ZZ_5]$. 

By the relations \eqref{R1} and \eqref{R2}, the differential ring
\begin{align*}
    \CC[{L^\bullet}^{\pm 1}][X^\bullet,\mathsf{D}^\bullet X^\bullet,{\mathsf{D}^\bullet}^2X^\bullet,\dots, Y^\bullet,\mathsf{D}^\bullet Y^\bullet,{\mathsf{D}^\bullet}^2Y^\bullet,\dots]
\end{align*}
is just the polynomial ring 
$$\CC[{L^\bullet}^{\pm 1}][X^\bullet,D^\bullet X^\bullet, {D^\bullet}^2 X^\bullet, Y^\bullet].$$
Denote this polynomail ring by
$$\mathds{F}^\bullet.$$

\

\section{Gromov-Witten invariants of local $\PP^4$}

\subsection{Formal Frobenious manifold.} Denote by $\mathsf{V}^{K\PP^4}$ the cohomology $H^*(\PP^4,\CC)$ of $\PP^4$. The genus $0$ $\mathsf{T}$-equivariant Gromov-Witten potential,
$$\mathsf{F}^{K\PP^4}_0(\gamma)=\sum_{d=0}^{\infty} t^d \sum_{r=0}^{\infty} \frac{1}{r!}\lan \gamma,\dots,\gamma\ran^{K\PP^4}_{0,d}\, , \, \gamma \in \mathsf{V}^{K\PP^4}$$
is a formal series in the ring $\mathsf{A}[[\mathsf{V}^*]]$ where
$$\mathsf{A}=\CC[[t]].$$
Note we do not have equivariant parameters since we use specialization \eqref{Specialization}.
The $\mathsf{T}$-equivariant genus $0$ potential $\mathsf{F}^{K\PP^4}_0$ defines a formal Frobenius manifold
$$(\mathsf{V}^{K\PP^4},\bullet,\eta)$$
at the origin of $\mathsf{V}^{K\PP^4}$.

\subsection{$\mathsf{R}$-matrix.}\label{RP2} For semisimple Gromov-Witten theoy with torus action, Givental proved reconstruction theorem using torus localization strategy, see \cite{SS,Book}. This method was applied to stable quotient theory of local $\PP^2$ in \cite{LP}.  The exactly same method in \cite{LP} yield the similar results for local $\PP^4.$ We here summarize the result for $(\mathsf{V}^{K\PP^4},\bullet,\eta)$.

Let 
\begin{align}\label{basis}
\{1,H,H^2,H^3,H^4\}\subset \mathsf{V}^{K\PP^4}
\end{align}
be a basis where $H$ is hyperplane class in $H^*(\PP^4)$.
Following the notation of \cite{KL}, 
we define series for the $K\PP^4$
geometry,


\begin{align*}
\langle \langle \gamma _1\psi  ^{a_1} , \ldots, \gamma _n\psi  ^{a_n} \rangle\rangle^{K\PP^4} _{0, n} 
= \sum _{d\geq 0} q^{d}
 \lan    \gamma _1\psi  ^{a_1} , \ldots, \gamma _n\psi  ^{a_n} \ran_{0, n, d}^{K\PP^4} \, ,
 \end{align*}
Define the series $\mathds{U}_i$ and $\mathds{S}_{i}$ by
\begin{align*}
    \mathds{U}^{K\PP^4}_i&=\lan\lan \phi_i,\phi_i\ran\ran^{K\PP^4}_{0,2}\,,\\
    \mathds{S}^{K\PP^4}_i(\gamma)&=e_i\lan\lan\frac{\phi_i}{z-\psi},\gamma\ran\ran^{K\PP^4}_{0,2}\,,
\end{align*}
where $e_i=-5\prod_{i=k}^4(1-\zeta^k)=-25$ is the $\mathsf{T}$-weight at the fixed point $p_i \in (K\PP^4)^{\mathsf{T}}$.

An evaluation of the series $\mathds{S}^{K\PP^4}_i(H^j)$ admits following asymtotic form
\begin{align*}
    \mathds{S}^{K\PP^4}_i(H^j)=e^{\frac{\mathds{U}^{K\PP^4}_i}{z}}\frac{(L^{K\PP^4})^j \zeta^{ij}}{C^{K\PP^4}_1\dots C^{K\PP^4}_j}\left( R_{j0}+R_{j1} \frac{z}{\zeta^i}+R_{j2}(\frac{z}{\zeta^i})^2+\dots \right)\,, \,\text{for},\ 0 \le j \le 4.
\end{align*}
The series $R_{jp}$ satisfy following system of equation.
\begin{align}\label{E1}
    \nonumber R_{1p+1}&=R_{0p+1}+\frac{D R_{0p}}{L}\, ,\\
    \nonumber R_{2p+1}&=R_{1p+1}+\left(\frac{DL}{L^2}-\frac{X}{L}\right)R_{1p}+\frac{D R_{1p}}{L}\, ,\\
    R_{3p+1}&=R_{2p+1}+\left(2\frac{DL}{L^2}-\frac{X}{L}-\frac{Y}{L}\right)R_{2p}+\frac{DR_{2p}}{L}\, ,\\
    \nonumber R_{4p+1}&=R_{3p+1}+\left(\frac{X}{L}+\frac{Y}{L}-2\frac{DL}{L^2}\right)R_{3p}+\frac{DR_{3p}}{L}\, ,\\
    \nonumber R_{0p+1}&=R_{4p+1}+\left(\frac{X}{L}-\frac{DL}{L^2}\right)R_{4p}+\frac{DR_{4p}}{L}\, .
\end{align}

\noindent Denote by $\tilde{\mathsf{R}}^{K\PP^4}(z)$ the matrix whose $(i,j)$-th component is $\sum_p R_{jp}(\frac{z}{\zeta^{i}})^p$. The $\mathsf{R}$-matrix $\mathsf{R}^{K\PP^4}(z)$ for $K\PP^4$ are given by
the following, see \cite{SS, LP}.

\begin{Prop}\label{RM1} We have

\begin{align*}
    \left[\mathsf{R}^{K\PP^4}(z)\right]_{ij}={\rm Exp}\left(-\sum_{k=1}^\infty \frac{N_{2k-1}}{2k-1}\frac{B_{2k}}{2k} (\frac{z}{\zeta^i})^{2k-1}\right)\left[\tilde{\mathsf{R}}^{K\PP^4}(z)\right]_{ij},
\end{align*}
where $N_{k}=(-\frac{1}{5})^k+\sum_{i=1}^4 (\frac{1}{1-\zeta^i})^k$.

\end{Prop}

\

\section{Gromov-Witten invariats of $[\CC^5/\ZZ_5]$}

\subsection{Fomal Frobenious manifold}
Denote by $\mathsf{V}^{[\CC^5/\ZZ_5]}$ the orbifold cohomology $H_{\text{orb}}^* (B\ZZ_5)$ of $B\ZZ_5$. For $\gamma \in \mathsf{V}^{[\CC^5/\ZZ^5]}$, the genus zero $\mathsf{T}$-equivariant Gromov-Witten potential,

\begin{align*}
    \mathsf{F}^{[\CC^5/\ZZ_5]}_0(\gamma)= \sum_{r=0}^\infty \sum_{m=0}^\infty \frac{1}{r!} \frac{1}{m!}\lan \gamma,\dots,\gamma,s \phi_1,\dots,s \phi_1 \ran^{[\CC^5/\ZZ_5]}_{0,r+m}
\end{align*}
is formal series in the ring $\mathsf{A}[[\mathsf{V}^*]]$ where
$$\mathsf{A}=\CC[[s]].$$
The $\mathsf{T}$-equivariant genus zero potential $\mathsf{F}^{[\CC^5/\ZZ_5]}_0$ define a formal Frobenius maifold
$$(\mathsf{V}^{[\CC^5/\ZZ_5]},\bullet,\eta)$$
at the origin of $\mathsf{V}^{[\CC^5/\ZZ_5]}$.

\subsection{$\mathsf{R}$-matirx} Here we summarize the result for $(\mathsf{V}^{[\CC^5/\ZZ_5]},\bullet,\eta)$. See apenxid for the proof.
\subsubsection{Frobenius structure}\label{FS}
We descibe the Frobenius structure on the orbifold cohomology  $H^*_{orb}([\CC^5/\ZZ_5])$ by following data.

\begin{itemize}
 \item{Metric:}
 In the basis $\{\phi_0,\phi_1,\phi_2,\phi_3,\phi_4 \}$, the metric $\eta$ is given by
 
 \[
   \mathbf{\eta}=\frac{1}{5}
  \left[ {\begin{array}{ccccc}
   1 & 0 & 0 & 0 & 0 \\
   0 & 0 & 0 & 0 & 1 \\
   0 & 0 & 0 & 1 & 0 \\
   0 & 0 & 1 & 0 & 0 \\
   0 & 1 & 0 & 0 & 0 \\
  \end{array} } \right]
\].

 \item{Quantum product:}
 \begin{align*}
     &\phi_0 \bullet \phi_0=\phi_0 &,&\,\,\,\,\,\,\,\,\, \phi_0 \bullet \phi_1=\phi_1\\
     &\phi_0 \bullet \phi_2=\phi_2 &,&\,\,\,\,\,\,\,\,\, \phi_0 \bullet \phi_3=\phi_3\\
     &\phi_0 \bullet \phi_4=\phi_4 &,&\,\,\,\,\,\,\,\,\, \phi_1 \bullet \phi_1= \frac{C_2}{C_1}\phi_2\\
     &\phi_1 \bullet \phi_2=-\frac{L^5}{C_1^3C_2^2} \phi_3 &,&\,\,\,\,\,\,\,\,\,  \phi_1 \bullet \phi_3=\frac{C_2}{C_1}\phi_4\\
     &\phi_1 \bullet \phi_4=\phi_0 &,&\,\,\,\,\,\,\,\,\, \phi_2 \bullet \phi_2=-\frac{L^5}{C_1^3C_2^2}\phi_4\\
     &\phi_2 \bullet \phi_3=\phi_0 &,&\,\,\,\,\,\,\,\,\, \phi_2 \bullet \phi_4=\frac{C_1}{C_2}\phi_1\\
     &\phi_3 \bullet \phi_3=-\frac{C_1^3C_2^2}{L^5}\phi_1 &,&\,\,\,\,\,\,\,\,\, \phi_3 \bullet \phi_4=-\frac{C_1^3C_2^2}{L^5}\phi_2\\
     &\phi_4 \bullet \phi_4=\frac{C_1}{C_2}\phi_3
 \end{align*}

\end{itemize}

\subsubsection{Canonical coordinate}

We normalize the basis $\{\phi_0,\phi_1,\dots,\phi_4\}$ by following.

\begin{align*}
    &\tilde{\phi}_0=\phi_0,\\
    &\tilde{\phi}_1=-\frac{C_1}{L}\phi_1,\\
    &\tilde{\phi}_2=\frac{C_1C_2}{L^2}\phi_2\\
    &\tilde{\phi}_3=\frac{L^2}{C_1C_2}\phi_3\\
    &\tilde{\phi}_4=-\frac{L}{C_1}\phi_4.
    \end{align*}


\noindent Then we have the following equations.

\begin{align*}
    \tilde{\phi}_i \bullet \tilde{\phi}_j = \tilde{\phi}_{i+j}.
\end{align*}
Here we use the convention $i+5=i$ for the subscript in the above equation.
\noindent The quantum product is semisimple and the canonical basis is given by the following.

\begin{align}\label{canbasis}
    e_\alpha=\frac{1}{5}\sum_i \zeta^{-\alpha i}\tilde{\phi}_i\, , \ \ \ \alpha=0,1,2,3,4,
\end{align}
where $\zeta = e^{\frac{2 \pi i}{5}}$. They satisfy the following equations.

\begin{align}\label{ss}
    e_{\alpha} \bullet e_{\beta}= \delta_{\alpha \beta}e_{\alpha}.
\end{align}
Define canonical coordinate $\{u^\alpha\}$ by following equation.

$$\sum_{\alpha=0}^4 e_\alpha d u^\alpha = \phi_1 ds \, ,$$
with initial conditions 

$$ u^\alpha|_{s=0}=0\,,\,\,\text{for} \, \alpha=0,1,\dots,4\,.$$
From \eqref{canbasis} and \eqref{ss} we easily obtain the following result.
\begin{Lemma}
 We have 
 \begin{align*}
     d u^\alpha = -\zeta^\alpha L \frac{d\psi}{\psi}.
 \end{align*}
\end{Lemma}

\noindent Denote the normalized canonical basis by 

\begin{align*}
    \tilde{e}_\alpha= \frac{e_\alpha}{\sqrt{\eta(e_\alpha,e_\alpha)}}=5 e_\alpha.
\end{align*}

\subsubsection{Transition matrix}

The transition matrix $\mathsf{\Psi}$ from flat coordinate to normalized canonical basis is given by

$$\mathsf{\Psi}_{\alpha i}=\eta(\tilde{e}_\alpha,\phi_i). $$

\noindent From \eqref{canbasis}, we can calculate $\mathsf{\Psi}$ explicitly as follows.

\[
   \mathsf{\Psi}=\frac{1}{5}
  \left[ {\begin{array}{ccccc}
    1 & -\frac{L}{C_1}        & \frac{L^2}{C_1C_2} & \frac{C_1C_2}{L^2}& -\frac{C_1}{L} \\
    1 & -\zeta\frac{L}{C_1}        & \zeta^2\frac{L^2}{C_1C_2} & \zeta^3\frac{C_1C_2}{L^2}& -\zeta^4\frac{C_1}{L} \\
    1 & -\zeta^2\frac{L}{C_1}        & \zeta^4\frac{L^2}{C_1C_2} & \zeta\frac{C_1C_2}{L^2}& -\zeta^3\frac{C_1}{L} \\
    1 & -\zeta^3\frac{L}{C_1}        & \zeta\frac{L^2}{C_1C_2} & \zeta^4\frac{C_1C_2}{L^2}& -\zeta^2\frac{C_1}{L} \\
    1 & -\zeta^4\frac{L}{C_1}        & \zeta^3\frac{L^2}{C_1C_2} & \zeta^2\frac{C_1C_2}{L^2}& -\zeta\frac{C_1}{L} 
  \end{array} } \right].
\]

\noindent Denote by $R^{k}_{ij}$ the $(i,j)$-component of $\mathsf{R}^{-1}\mathsf{\Psi}$.
From the flatness equation \eqref{FSM} we obtain the following results.


\begin{align}\label{E2}
    \nonumber &D R^{k-1}_{i0}+L R^k_{i0}\zeta^i+C_1 R^k_{i1}=0 \,,\\
    \nonumber &D R^{k-1}_{i1}+L R^k_{i1}\zeta^i+C_2 R^k_{i2}=0 \,,\\
    &D R^{k-1}_{i2}+L R^k_{i2}\zeta^i-\frac{L^5}{C_1^2C_2^2} R^k_{i3}=0 \,,\\
    \nonumber &D R^{k-1}_{i3}+L R^k_{i3}\zeta^i+C_2 R^k_{i4}=0 \,,\\
    \nonumber &D R^{k-1}_{i4}+L R^k_{i4}\zeta^i+C_1 R^k_{i0}=0 \, .
\end{align}

\noindent The following normalizations are very useful to solve \eqref{E2}.
\begin{align}\label{normalization}
  \nonumber&R^k_{i0}= \tilde{R}^k_{i0}\zeta^{-ki}\,, \\
  \nonumber&R^k_{i1}=-\frac{L}{C_1} \tilde{R}^k_{i1} \zeta^{i-ki}\,, \\
  &R^k_{i2}=\frac{L^2}{C_1C_2} \tilde{R}^k_{i2} \zeta^{2i-ki}\,, \\
  \nonumber&R^k_{i3}=\frac{C_1C_2}{L^2} \tilde{R}^k_{i3} \zeta^{3i-ki}\,, \\
  \nonumber&R^k_{i4}=-\frac{C_1}{L} \tilde{R}^k_{i4} \zeta^{4i-ki} \,.
\end{align}

\noindent We solve the above system of equations with initial condition:
\begin{align*}
    \tilde{R}^k_{ij}(0)=0 .
\end{align*}

\noindent Denote by $\tilde{\mathsf{R}}^{[\CC^5/\ZZ_5]}$ the matrix whose $(i,j)$-th component is $$\sum_k \tilde{R}^k_{ij} (\frac{z}{\zeta^i})^k$$.

Applying orbifold quantum Riemann-Roch theorem (Theorem 4.2.1 in \cite{Ts}) to our case, we can determine the constant terms in $\mathsf{R}^{[\CC^5/\ZZ_5]}$.

\begin{Prop}\label{RM2} We have following formular for $\mathsf{R}$-matrix of $[\CC^5/\ZZ_5]$ after restriction $\psi=0$.
 \begin{align*}
     \left[\mathsf{R}^{[\CC^5/\ZZ_5]}\right]_{ij}|_{\psi=0}={\rm Exp}\left(5\sum_{k=1}^\infty (-1)^{k+1} \frac{B_{5k+1}(i/5)}{5k+1} \frac{z^{5k}}{5k(\zeta^j)^{5k}}\right),
 \end{align*}
where $B_k(x)$ is $k$-th Bernoulli polynomail. 
 
\end{Prop}

\noindent By Theorem \ref{Classification}, we obtain the following result.

\begin{Cor}\label{RM3}The $\mathsf{R}$-matrix for $[\CC^5/\ZZ_5]$ equal to following form.
\begin{align*}
    \mathsf{R}^{[\CC^5/\ZZ_5]}={\rm Exp}\left(5\sum_{k=1}^\infty (-1)^{k+1} \frac{B_{5k+1}(i/5)}{5k+1} \frac{z^{5k}}{5k(\zeta^j)^{5k}}\right) \tilde{\mathsf{R}}^{[\CC^5/\ZZ_5]}.
\end{align*}
\end{Cor}

\

\section{Crepant resolution conjecture}

We state here the crepant resolution conjecture for all genera and prove the conjecture for genus $2,3$.

By the relations \eqref{R1} and \eqref{R2}, we have 
\begin{align}\label{R3}
    \mathsf{R}^{K\PP^4} \in \text{Mat}_{5 \ti5}(\mathds{F}^{K\PP^4}[[z]]) ,\\
    \nonumber \mathsf{R}^{[\CC^5/\ZZ_5]} \in \text{Mat}_{5 \ti 5}(\mathds{F}^{[\CC^5/\ZZ_5]}[[z]]) .
\end{align}
Define transformation $\mathds{T}$ from $\mathds{F}^{K\PP^4}$ to $\mathds{F}^{[\CC^5/\ZZ_5]}$ by
\begin{align*}
    &\mathds{T}(L)=-\frac{L}{5}\\
    &\mathds{T}(X)=-\frac{X}{5}\\
    &\mathds{T}(DX)=\frac{DX}{5^2}\\
    &\mathds{T}(D^2X)=-\frac{D^2X}{5^3}\\
    &\mathds{T}(Y)=-\frac{Y}{5}.
\end{align*}
\noindent In above we omitted the obvious upper subscript.

\begin{Conj}\label{CRC}
The transformation $\mathds{T}$ satisfies the following. 
\begin{align*}\mathds{T}(\mathsf{R}^{K\PP^4})=\mathsf{R}^{[\CC^5/\ZZ_5]}.
\end{align*}
\end{Conj}

Denote by $F^\bullet_g$ the genus $g$ series for the corresponding theories,

$$F^\bullet_g=\lan\lan \ran\ran^\bullet_{g,0}\,, \,\, \text{for}\, \bullet=K\PP^4 \,\text{or}\, [\CC^5/\ZZ_5]\,.$$  From Theorem \ref{Classification} and \eqref{R3} we also have the following results.
\begin{align*}
    F_g^{K\PP^4}\in \mathds{F}^{K\PP^4},\,\,\,\,
    F_g^{[\CC^5/\ZZ_5]} \in \mathds{F}^{[\CC^5/\ZZ_5]}.
\end{align*}
By Theorem \ref{Classification}, Conjecture \ref{CRC} is equivalent to the following. 

\begin{Conj}\label{Conj1}
 For $g \ge 2$, the transformation $\mathds{T}$ satisfies the following.
\begin{align*}
\mathds{T}(F_g^{K\PP^4})=F_g^{[\CC^5/\ZZ_5]}.
\end{align*}

\end{Conj}

For an element $C\in \mathds{F}^{K\PP^4}$, let $C^+$ be the non-negative degree part with respect to $L$ in $C$. Define a function 
$$M : \mathds{F}^{K\PP^4} \rightarrow \QQ$$
by

$$M(C):= C^+|_{L=0,X=-\frac{1}{5},Y=-\frac{1}{5},DX=0,D^2X=0}\,.$$
Define $a^i_k \in \QQ$ by
\begin{align*}
    a^i_k:= M(R_{ik})\,,
\end{align*}
where $R_{ik} \in \mathds{F}^{K\PP^4}$ was defined in Section 4.2.

\begin{Prop}\label{Prop1} Conjecture \ref{CRC} is equivalent to following equations for $0 \le i \le 4$.
\begin{align*}
    {\rm Exp}\left(-\sum_{k=1}^\infty \frac{N_{2k-1}}{2k-1}\frac{B_{2k}}{2k} z^{2k-1}\right)\sum_{k=0}^\infty a^i_k z^k = {\rm Exp}\left(5\sum_{k=1}^\infty(-1)^{k+1} \frac{B_{5k+1}(i/5)}{5k+1}\frac{z^{5k}}{5k}\right)
\end{align*}

\end{Prop}

\begin{proof}
 It is easy to check that under the transformation $\mathds{T}$ the system of equations \eqref{E1} is equivalent to \eqref{E2}. Hence the solutions of two systems of equations differ by constants. In other words \begin{align}\label{RE}
 \mathds{T}(\mathsf{R}^{K\PP^4})=(\sum_{i=0}^\infty \mathsf{A}_i z^i)\mathsf{R}^{[\CC^5/\ZZ_5]},
 \end{align}
 for some constant matirice $\mathsf{A}_i$. Here we consider each coefficient of $z$ in \eqref{RE} as an element in $\mathds{F}^{[\CC^5/\ZZ_5]}$.  From Proposition \ref{RM1} and Proposition \ref{RM2} we obtain the equations in the proposition by restricting \eqref{RE} to $L=0$.  

\end{proof}

We checked the equation in Proposition \ref{Prop1} up to degree of $z$ equal to 6. We give some computational details in Appendix 7.2. Therefore we obtain following result.

\begin{Thm}
Conjecture \ref{Conj1} is true for genus $2,3$.
\end{Thm}

\

\section{Appendix}

\subsection{Qauntum product of $[\CC^5/\ZZ_5]$} To compute the quantum product of $[\CC^5/\ZZ_5]$, we need to know the genus $0$ three point correlator.

\begin{Lemma} The genus $0$ three point correlators are as follows.
\begin{align*}
    &\lan\lan \phi_0,\phi_0,\phi_0\ran\ran^{[\CC^5/\ZZ_5]}_{0,3} = \frac{1}{5} \,\,,\,\, &\lan\lan \phi_0,\phi_1,\phi_4\ran\ran^{[\CC^5/\ZZ_5]}_{0,3} = \frac{1}{5} \,\, , \,\, \\
    &\lan\lan \phi_0,\phi_2,\phi_3\ran\ran^{[\CC^5/\ZZ_5]}_{0,3} = \frac{1}{5} \,\,,\,\,
    &\lan\lan \phi_1,\phi_1,\phi_3\ran\ran^{[\CC^5/\ZZ_5]}_{0,3} = \frac{1}{5} \frac{C_2}{C_1} \,\, ,\\ &\lan\lan \phi_1,\phi_2,\phi_2\ran\ran^{[\CC^5/\ZZ_5]}_{0,3} = -\frac{1}{5}\frac{L^5}{C_1^3 C_2^2}\,\,,\,\,
    &\lan\lan \phi_2,\phi_4,\phi_4\ran\ran^{[\CC^5/\ZZ_5]}_{0,3} = \frac{1}{5} \frac{C_1}{C_2} \,\, ,\\ 
    &\lan\lan \phi_3,\phi_3,\phi_4\ran\ran^{[\CC^5/\ZZ_5]}_{0,3} = -\frac{1}{5}\frac{C_1^3 C_2^2}{L^5}.
\end{align*}

For other choices of insertions, the genus 0 three point correlators are equal to zero.
\end{Lemma}

\begin{proof}
It was proven in \cite{CCIT} that $I^{[\CC^5/\ZZ_5]}$ is on the Lagrangian cone $\mathcal{L}^{[\CC^5/\ZZ_5]}$ encoding the genus $0$ Gromov-Witten theory of $[\CC^5/\ZZ_5]$, see \cite{CG,SF} for the definition of the Lagrangian cone. By the now standard properties of the Lagrangian cone $\mathcal{L}^{[\CC^5/\ZZ_5]}$, one can show the following, see for example \cite{KL,RR}. 
\begin{align}\label{BF}
\mathds{S}^{[\CC^5/\ZZ_5]}(s,z)(\phi_k) = \frac{(\mathsf{M}^{[\CC^5/\ZZ^5]})^k(I^{[\CC^5/\ZZ_5]}(\psi,z))}{C^{[\CC^5/\ZZ_5]}_k}\,,\,\,\,\, \text{for}\,\,k=0,1,2,3,4 , 
\end{align}
where the $\mathds{S}$-operator for $[\CC^5/\ZZ_5]$ is defined as usual by 

\begin{align*}
\mathds{S}^{[\CC^5/\ZZ_5]}(s,z)(\gamma) = \sum_i \phi^i \lan\lan \frac{\phi_i}{z-\psi},\gamma  \ran\ran^{[\CC^5/\ZZ_5]}_{0,2} \,\, , \,\, \text{for}\, \gamma \in H^*_{\text{orb}} ([\CC^5/\ZZ_5]).
\end{align*}
Observe that $I$-function has following expansion,

$$I^{[\CC^5/\ZZ_5]}=\phi_0+\frac{I_1 \phi_1}{z}+\frac{I_2 \phi_2}{z^2}+\frac{I_3 \phi_3}{z^3}+\frac{I_4 \phi_4}{z^4} +\mathsf{O}(\frac{1}{z^5})\,.$$
Then equation \eqref{BF} immediately yield

\begin{align*}
\lan\lan \phi_0,\phi_1,\phi_4\ran\ran^{[\CC^5/\ZZ_5]}_{0,3} &= \frac{1}{5} \, , \\ 
\lan\lan \phi_1,\phi_1,\phi_3\ran\ran^{[\CC^5/\ZZ_5]}_{0,3} &= \frac{1}{5} \frac{C_2}{C_1} \,, \\
\lan\lan \phi_1,\phi_2,\phi_2\ran\ran^{[\CC^5/\ZZ_5]}_{0,3} &= -\frac{1}{5}\frac{L^5}{C_1^3 C_2^2}\,. 
\end{align*}
Recall the property of Frobenius manifold,
\begin{align}\label{P4}
\eta(X\bullet Y,Z)=\lan\lan X,Y,Z \ran\ran^{[\CC^5/\ZZ_5]}_{0,3} \,\, , \, \text{for}\, X,Y,Z \in H^*_{\text{orb}} ([\CC^5/\ZZ_5])\,.
\end{align}
Then the other results in the Lemma follows from associativity of quantum product.

\end{proof}

Combining the result of above Lemma with \eqref{P4}, one can compute the quantum product of $[\CC^5/\ZZ_5]$ as in Section \ref{FS}.

\subsection{Some computational results for local $\PP^4$.} We here summarize the method of \cite{ZaZi} to solve the Picard-Fuchs equation of local $\PP^4$. Throughout this subsection, we omitt the upper subscript $K\PP^4$ for $L,D$  which were defeined in Section \ref{BHS}.

Define differential operators $\mathcal{L}_k$ by

\begin{align*}
    &\mathcal{L}_1 := (1-L^5) + 5 D \,, \\
    &\mathcal{L}_2 := \frac{1}{5}(4-L^5-3L^{10})+(6-6L^5)D+10D^2\,,\\
    &\mathcal{L}_3 := \frac{1}{25}(12-7L^5-2L^{10}-3L^{15})+\frac{1}{5}(22-13L^5-9L^{10})D+(12-12L^5)D^2+10D^3\,,\\
    &\mathcal{L}_4:=\frac{1}{625}(120-103L^5+61L^{10}-144L^{15}+66L^{20})\\&+\frac{1}{25}(50-41L^5-3L^{10}-6L^{15})D+(7-\frac{29}{5}L^5-\frac{6}{5}L^{10})D^2+(10-10L^5)D^3+5D^4\,,\\
    &\mathcal{L}_5:=-\frac{24}{625}(L^5-1)-\frac{274}{625}(L^5-1)D-\frac{9}{5}(L^5-1)D^2-\frac{17}{5}(L^5-1)D^3\\
    &+(3-3L^5)D^4+D^5 \,.
\end{align*}
The following proposition is proved in \cite{ZaZi}. Since the setting in \cite{ZaZi} is slightly different from ours, we need some normalizations.

\begin{Prop}
The series 
$$Q_p:=L R_{1p}$$ satisfies following recursive differential equations.

\begin{align*}
    \mathcal{L}_1(Q_p)+\frac{1}{L}\mathcal{L}_2(Q_{p-1})+\frac{1}{L^2}\mathcal{L}_3(Q_{p-2})+\frac{1}{L^3}\mathcal{L}_4(Q_{p-3})+\frac{1}{L^4}\mathcal{L}_5(R_{p-4})=0\,.
\end{align*}
\end{Prop}

\noindent Here $R_{ip}$ is the $q$-power series defined in Section \ref{RP2}.
The first few solutions with initial conditions 
\begin{align*}
&Q_0=L\,,\\
&Q_p|_{q=0}=0\,\,\,\,\text{for}\,\,p \ge 1 \,,
\end{align*}
can be calculated as follows.

\begin{align*}
    &R_{10}:=1\,,\\
    &R_{11}=\frac{3}{20}(1-L^4)\,,\\
    &R_{12}=\frac{9}{800}(1-L^4)^2\,,\\
    &R_{13}=\frac{1}{80000}(269+4288L^2-135L^4-16128L^7+135L^8+11571L^{12})\,,\\
    &R_{14}=\frac{1}{6400000}(2823+137216L+51456L^2-3228L^4-2041088L^6\\
    &-193536L^7+810L^8+4322304L^{11}+138852L^{12}-2415609L^{16})\,,\\
    &R_{15}=\frac{3}{128000000}(50532+137216L+25728L^2-2823L^4-4634624L^5\\
    &-2041088L^6-96768L^7+1614L^8+23404672L^{10}+4322304L^{11}\\
    &+69426L^{12}-34732544L^{15}-2415609L^{16}+15911973L^{20})\,,\\
    &R_{16}=\frac{1}{25600000000}(4564757+6174720L+4613888L^2+4493426178L^4\\
    &-417116160L^5-91848960L^6-17353728 L^7+127035L^8-47045380096 L^9\\
    &+2106420480L^{10}+194503680L^{11}+12450396L^{12}+132709674240L^{14}\\
    &-3125928960L^{15}-108702405L^{16}-143147676672L^{19}+1432077570L^{20}+52989974037L^{24})
\end{align*}
Using \eqref{E1}, we can also calculate the other series 
$$R_{ip}\,\,\,\,\, \text{for}\,\,i=0,2,3,4.$$

\end{document}